\documentclass[11pt,a4paper]{amsart}
\usepackage[all]{xy}
\usepackage{amscd}
\usepackage{amssymb}

\newtheorem{theorem}{Theorem}[section]
\newtheorem{lemma}[theorem]{Lemma}

\newtheorem{prop}[theorem]{Proposition}

\newtheorem*{ThA}{Theorem A}
\newtheorem*{ThB}{Theorem B}
\theoremstyle{definition}
\newtheorem{defin}[theorem]{Definition}

\theoremstyle{remark}
\newtheorem*{erem}{Remark}

 \def\={\setminus}  
 \def\lenght{\mathop\mathrm{lenght}\nolimits}
 
\def\md#1{#1\mbox{-}\mathrm{mod}}
\def\pr#1{#1\mbox{-}\mathrm{pro}}

\def\bp{\bullet}    \def\iso{\simeq}

\def\iff{if and only if }

\def\setsuch#1#2{\left\{\,#1\,|\,#2\,\right\}}

\def\rad{\mathop\mathrm{rad}}

\def\End{\mathop\mathrm{End}}

\def\im{\mathop\mathrm{Im}}
\def\8{\infty}      \def\+{\oplus}
\def\*{\otimes}

\def\al{\alpha} \def\be{\beta}  
       \def\la{\lambda}
     
 \def\Si{\Sigma}

\def\bA{\mathbf A}  \def\bB{\mathbf B}
  
\def\fR{\mathbf r}  
	\def\bL{\mathbf L}
\def\bW{\mathbf W}	\def\bD{\mathbf D}

\def\mN{\mathbb N}  
\def\mA{\mathbb A}

 \def\kD{\mathcal D}
\def\kK{\mathcal K} 
 \def\kQ{\mathcal Q}
\def\kP{\mathcal P}     
 
\def\kJ{\mathcal J} \def\kI{\mathcal I}

    \def\dP{\mathfrak P}

      \def\sJ{\mathsf J}

\def\sR{\mathsf R}

\def\tA{\widetilde{A}}     
     
\def\Mk{\Bbbk}
\def\bA{\mathbf A}	\def\bS{\mathbf S}
 \def\bH{\mathbf H}	\def\wQ{\widehat{\Mk\mathcal Q}}
 \def\Mat{\mathop\mathrm{Mat}\nolimits}
 \def\xx{\times}	\def\bJ{\mathbf J}

 \def\kxy{\Mk\left<x,y\right>}

\begin{document}
\author{Viktor  Bekkert,  Yuriy  Drozd  \and Vyacheslav Futorny}
\title{Derived tame local and two-point algebras}
\address{Departamento de Matem\'atica, ICEx, Universidade Federal de Minas
Gerais, Av.  Ant\^onio Carlos, 6627, CP 702, CEP 30123-970, Belo
Horizonte-MG, Brasil} \email{bekkert@mat.ufmg.br}
\address{Institute of Mathematics, National Academy of Sciences of Ukraine, Te\-reschenkivska 3, 01601 Kiev, Ukraine}
\email{drozd@imath.kiev.ua}
\address{Instituto de Matem\'atica e Estat\'\i stica,
Universidade de S\~ao Paulo, Caixa Postal 66281, S\~ao Paulo, CEP
05315-970, Brasil} \email{futorny@ime.usp.br}
\subjclass[2000]{Primary: 16G60, 16G70 secondary: 15A21, 16E05, 18E30}

\begin{abstract}
{We determine derived representation type of  complete finitely generated local
and two-point algebras over an algebraically closed field.}
\end{abstract}
\maketitle

\section*{Introduction}

We consider finitely generated unital (associative) algebras  over an algebraically closed  field $\Mk$.

One of the  main problems in the representation theory of algebras is a classification of indecomposable finitely generated modules.
The dichotomy theorem \cite{d0} divides all finite dimensional algebras according to their {\em representation type} into {\em tame} and {\em wild}.
In the case of tame algebras a classification of indecomposable modules is relatively easy, for each dimension $d$ they admit a parametrization of $d$-dimensional indecomposable modules by a finite number of $1$-parameter families. The situation is much more complicated in the case of wild algebras. This
 singles out the problem of establishing the representation type of a given algebra. The answer is fully known for
complete local algebras (those algebras whose quiver contains a single vertex)  \cite{bod,br,d,ge2,gp,hr,r0,r} and for
 finite dimensional two-point algebras  \cite{bg,bhan,dg,g,ge1,han,hm}.
In the case of infinite dimensional two-point algebras the problem is still open,  except the pure noetherian algebras \cite{d1}.

During the last years there has been an active study of derived categories. In particular, a notion of {\em derived representation type} was introduced for finite dimensional algebras \cite{gk}. The tame-wild dichotomy for  derived categories over finite dimensional algebras was established in \cite{bd}. The structure of the derived category is known for a few classes of finite dimensional algebras (e.g. \cite{bm,bud,h,hapr}).

On the other hand, certain infinite dimensional algebras and their derived categories play an important role in applications, in particular in the study of singularities of projective curves (cf. \cite{bud1}).
First results  on derived representation type in the infinite dimensional case were obtained in \cite{bud}.

In the present paper we determine  representation type of the bounded
derived category
of finitely generated modules
over  finitely generated complete local
and two-point algebras.

Our main results are the following classification theorems

\begin{ThA}
Let $\bA$ be a complete local algebra over
algebraically closed field $\Mk$. Then $\bA $ is derived tame if
and only if $\bA$ is isomorphic to one of the following algebras:
\begin{itemize}
\item $\bL_1= \Mk.$
\item The \emph{algebra of dual numbers} $\bL_2=\Mk[x]/(x^{2}).$
\item The power series algebra $\bL_3=\Mk [[x]].$
\item $\bL_4=\Mk [[x,y]]/(xy)$ - the local ring of a \emph{simple node} of an algebraic curve over a field $\Mk$.
\item The \emph{dihedral algebra} $\bL_5=\Mk \left<\left< x,y\right>\right> /(x^{2}, y^{2})$.
\end{itemize}

Moreover, the first algebra is derived finite and the second  and
third are derived discrete.
\end{ThA}

\begin{ThB}
Let $\bA$ be the completion of a two-point algebra $\Mk\kQ/\kI$  over
algebraically closed field $\Mk$. Then

\medskip
\noindent(1) The following conditions are equivalent:

\medskip
(i) $\bA $ is derived tame.

\medskip
(ii) $\bA $ is either a gentle algebra  or a nodal non-gentle algebra (see Section~\ref{s2} for
definitions) or one
of the algebras $\bD_1$, $\bD_2$  (see Section~\ref{s24}).

\medskip
(iii) $\bA $ is isomorphic to one of the algebras from Table~\ref{table_gentle}  or to the algebra $(9)$ from Table~\ref{table_nodal} or to one
of the algebras $\bD_1$, $\bD_2$ %
\footnote{Note that all algebras from Table 1, except (9), are also in Table 2}.

\bigskip
\noindent(2) $\bA $ is derived discrete if and only if $\bA $ is
isomorphic to one of the algebras $(1), (3)-(5), (10)-(13), (16)-(17)$ from Table~\ref{table_gentle}.

\medskip
\noindent(3) $\bA $ is derived finite if and only if $\bA $ is
isomorphic to the algebra $(1)$ from Table~\ref{table_gentle}.

\end{ThB}

\begin{erem}
\begin{itemize}
\item Our definition (Definition~\ref{defgentlealg}) of a gentle algebra is slightly different from the standard one: we do not require the finite dimensionality of
such algebra but instead require its completeness.

\item Note that the classes of nodal and gentle algebras are not disjoint. In the case of two-point algebras there exists exactly one (up to isomorphism) nodal algebra which is not gentle. Obviously, there are many gentle algebras which are not nodal.
\end{itemize}
\end{erem}

The structure of the paper is as follows. In Section~\ref{s1}
preliminary results about derived categories and derived representation type are given.

In Section~\ref{s2}  we recall the definitions of nodal and gentle algebras, and classify
all such algebras in  local and two-point cases.

In Section~\ref{s3} we prove Theorems A and B.

\smallskip\emph{Acknowledgements.} This investigation was accomplished during the visit
 of the second author to the University o S\~ao Paulo supported by Fapesp (processo 2007/05047-4).
 The second author was also partially supported by INTAS Grant 06-1000017-9093.
The third author was partially supported by Fapesp (processo 2005/60337-2) and CNPq (processo 301743/2007-0).

\section{Derived representation type}\label{s1}

% We consider categories and algebras over a fixed algebraically closed field $\Mk$.

 Let $\bA$ be a \emph{semi-perfect} \cite{b} associative finitely generated  $\Mk$-algebra.
We denote by  $\md{\bA}$  the category of left finitely
generated $\bA$-modules and by $\kD(\bA)$  the derived category $\kD^{b}(\md{\bA})$ of
bounded complexes over $\md{\bA}$. As usually, it can be identified with the
homotopy category $\kK^{-,b}(\pr{\bA})$ of (right bounded) complexes of (finitely generated) projective
$\bA$-modules with bounded cohomologies. Since $\bA$ is semi-perfect, each complex from $\kK^{-,b}(\pr{\bA})$ is homotopic to a \emph{minimal}
one, i.e. to  a complex $C_{\bullet}=(C_n, d_n)$ such that $\im d_n\subseteq \rad C_{n-1}$ for all $n$. If $C_{\bullet}$
and $C^{\prime}_{\bullet}$ are two minimal complexes, they are isomorphic in $\kD(\bA)$ if and only if they are isomorphic as complexes.
Moreover, any morphism $f: C_{\bullet}\to C^{\prime}_{\bullet}$ in $\kD(\bA)$ can be presented by a morphism of complexes, and
$f$ is an isomorphism if and only if the latter one is. We denote by $\kP_{\min}(\bA)$ the category
 of minimal right bounded complexes of  (finitely generated) projective $\bA$-modules with bounded cohomologies.

Let $A_1, A_2, \ldots, A_t$ be all pairwise non-isomorphic indecomposable projective $\bA$-modules (all of them are direct summands of $\bA$).
If $P$ is a finitely generated projective $\bA$-module, it uniquely decomposes as $$P=\bigoplus_{i=1}^{t}p_iA_i.$$ Denote by $\fR(P)$ the
vector $(p_1,p_2,\ldots,p_t)$. The sequence $$(\dots,\fR(P_n),\fR(P_{n-1}),\dots)$$ (it
 has only finitely many nonzero entries)
is called the \emph{vector rank} $\fR_\bp(P_\bp)$
 of a bounded complex $P_\bp$ of projective $\bA$-modules.

The following definition is analogous to the definitions of \emph{derived tame} and \emph{derived wild} type for finite dimensional algebras \cite{bd}.
  \begin{defin}\label{tw}
  \begin{enumerate}
 \item
  We call a \emph{rational family} of bounded minimal complexes over $\bA$ a bounded complex $(P_\bp,d_\bp)$
 of finitely generated projective $\bA\*\sR$-modules, where $\sR$ is a \emph{rational algebra},
 i.e. $\sR=\Mk[t,f(t)^{-1}]$ for a nonzero polynomial $f(t)$, and $\im d_n\subseteq\sJ P_{n-1}$, where $\sJ=\rad\bA$.
 For a  rational family $(P_\bp,d_\bp)$ we define the complex $P_\bp(m,\la)=(P_\bp\*_\sR\sR/(t-\la)^m,d_\bp\*1)$ of projective $\bA$-modules, where $m\in\mN,\,\la\in\Mk,\,f(\la)\ne0$. Set $\fR_\bp(P_\bp)=
 \fR_\bp(P_\bp(1,\la))$ ($\fR_\bp$ does not depend on $\la$).
 \item
  We call an algebra $\bA$ \emph{derived tame} if there is a set $\dP$ of rational families of bounded complexes over
 $\bA$ such that:
  \begin{enumerate}
 \item  For each vector rank $\fR_\bp$ the set $\dP(\fR_\bp)=\setsuch{P_\bp\in\dP}{\fR_\bp(P_\bp)=\fR_\bp}$ is finite.
 \item
  For each vector rank $\fR_\bp$ all indecomposable complexes $(P_\bp,d_\bp)$ of projective $\bA$-modules
 of this vector rank, except finitely many isomorphism classes, are isomorphic to $P_\bp(m,\la)$ for some $P_\bp\in\dP$
 and some $m,\la$.
\end{enumerate}
 The set $\dP$ is called a \emph{parameterizing set} of $\bA$-complexes.
 \item
  We call an algebra $\bA$ \emph{derived wild} if there is a bounded complex $(P_\bp,d_\bp)$ of projective modules over
$\bA\*\Si$, where $\Si$ is the free $\Mk$-algebra in 2 variables,
such that $\im d_n\subseteq\sJ P_{n-1}$ and, for any finite dimensional $\Si$-modules $L,L'$,
  \begin{enumerate}
\item   $P_\bp\*_\Si L\iso P_\bp\*_\Si L'$ \iff $L\iso L'$.
 \item
  $P_\bp\*_\Si L$ is indecomposable \iff so is $L$.
\end{enumerate}
\end{enumerate}
 \end{defin}

 Note that, according to these definitions, every \emph{derived discrete} (in particular, \emph{derived finite} \cite{bm}) algebra
 \cite{v} is derived tame (with the empty set  $\dP$).

It is proved in \cite{bd} that every finite dimensional algebra over an
algebraically closed field is either derived tame or derived wild.

\section{Some classes of algebras}\label{s2}

\subsection{Quivers with relations}\label{s21}

A \emph{quiver} $\kQ$ is a tuple $(\kQ_0,\kQ_1,\mathfrak{s},
\mathfrak{t})$ consisting of a set $\kQ_0$ of \emph{vertices}, a
set $\kQ_1$ of \emph{arrows}, and maps $\mathfrak{s},\mathfrak{t}: \kQ_1
\rightarrow \kQ_0$ which specify the \emph{starting} and \emph{ending} vertices.
A \emph{path} $p$ in $\kQ$ of length $\ell(p) = n \geq 1$ is a
sequence of arrows $a_n,\dots,a_1$ such that
$\mathfrak{s}(a_{i+1}) = \mathfrak{t}(a_i)$ for $1 \leq i < n$.
Note that we write paths from  right to left for convenience.
 For a path $p$ set $\mathfrak{s}(p) = \mathfrak{s}(a_1)$ and $\mathfrak{t}(p) =
\mathfrak{t}(a_n)$. Then the concatenation $p^{\prime}p$ of two paths $p$,
$p^{\prime}$ is defined in the natural way whenever $\mathfrak{s}(p')=\mathfrak{t}(p)$.
Every vertex $i \in \kQ_0$ determines a path $e_i$ (of length $0$) with $\mathfrak{s}(e_i) = i$ and
$\mathfrak{t}(e_{i}) = i$.
A quiver $\kQ$ determines  the path algebra $\Mk\kQ$, which has an
$\Mk$-basis consisting of the paths of $\kQ$ with multiplication
given by the concatenation of paths.
The algebra $\Mk\kQ$ is finite-dimensional precisely when $\kQ$ does not contain an oriented cycle.
An ideal $\kI\subseteq \Mk\kQ$ is called \emph{admissible} if $\kI\subseteq \rad^{2}(\Mk\kQ)$ where $\rad(\Mk\kQ)$ is the radical of the algebra $\Mk\kQ$.
It is well-known that if $\Mk$ is algebraically closed, any finite-dimensional $\Mk$-algebra is Morita
equivalent to a quotient $\Mk\kQ/\kI$ where $\kI$ is an admissible ideal.
By a slight abuse of notation we identify paths in the quiver $\kQ$ with their cosets in $\Mk\kQ/\kI$.

\subsection{Nodal algebras}\label{s22}

\begin{defin}\label{noddef} A semi-perfect  noetherian algebra $\bA$  is called \emph{nodal} if it is
\emph{pure noetherian} (i.e. has no minimal ideals), and there is a hereditary algebra $\bH\supseteq\bA$, which
is semi-perfect and pure noetherian such that
\begin{itemize}
\item $\rad \bA=\rad \bH$.
\item $\lenght_{\bA}(\bH\otimes_{\bA}U)\leq 2$ for every simple left $\bA$-module $U$.
\item $\lenght_{\bA}(V\otimes_{\bA}\bH)\leq 2$ for every simple right $\bA$-module $V$.
\end{itemize}
\end{defin}

It was shown in \cite{d1} that nodal algebras are the only pure noetherian algebras such that the classification  of their modules of finite
length is tame (all others being wild).

\begin{prop}\label{nodal}

(1) Let $\bA$ be a  local $\Mk$-algebra.
Then $\bA$ is nodal if and only if it is isomorphic to one of the following algebras:
\begin{itemize}
\item The  algebra $\Mk [[x]]$ of power series.
\item The local ring $\Mk [[x,y]]/(xy)$  of a \emph{simple node} of an algebraic curve over $\Mk$.
\item The \emph{dihedral algebra} $\Mk \left<\left< x,y\right>\right> /(x^{2}, y^{2})$.
\end{itemize}

\medskip
(2) Let $\bA =\Mk\kQ/\kI$ be a  two-point algebra.
Then the following conditions are equivalent:

\begin{itemize}
\item $\bA $ is nodal.
\item $\bA $ is isomorphic to the completion of one of the algebras from Table~\ref{table_nodal} below.
\end{itemize}
\end{prop}

\begin{table}[ht]
\caption{Nodal two-point algebras}\label{table_nodal}
\renewcommand\arraystretch{1.5}
\noindent\[
\begin{array}{|c|c|}
\hline
\xymatrix{
1 \ar@/^1pc/[r]^{a}  &2 \ar@/^1pc/[l]_{b}
}&\xymatrix{
1 \ar@/^1pc/[r]^{b}  \ar@/^2pc/[r]^{a} &2 \ar@/^1pc/[l]^{c} \ar@/^2pc/[l]^{d}
}\\
(1)\hspace{10pt} \kI=0&(2)\hspace{10pt}\kI=\left<ca, db, ac, bd\right>\\
&(3)\hspace{10pt} \kI=\left<ca, db, bc, ad\right>\\
\hline
\xymatrix{
1 \ar@/^1pc/[r]^{b} \ar@(ul,dl)[]_{a}  &2 \ar@/^1pc/[l]^{c}
}&\xymatrix{
1 \ar@/^1pc/[r]^{c} \ar@(ul,dl)[]_{a}  &2 \ar@/^1pc/[l]^{d} \ar@(ur,dr)[]^{b}
}\\
(4)\hspace{10pt} \kI=\left<a^{2}, bc\right>&(6)\hspace{10pt}\kI=\left<a^{2}, b^{2}, dc, cd\right>\\
(5)\hspace{10pt} \kI=\left<ba, ac\right>&(7)\hspace{10pt}\kI=\left<a^{2}, db, bc, cd\right>\\
&(8)\hspace{10pt}\kI=\left<ca, db, bc, ad\right>\\
&(9)\hspace{10pt}\kI=\left<a^{2}-dc, b^{2}-cd,ca-bc,db-ad\right>\\
\hline
\end{array}
\]
\end{table}

\begin{proof}
 All algebras under
 consideration  are of the form $\bA=\wQ/\kI$ for some finite connected quiver
 $\kQ$ and some admissible ideal $\kI\subseteq\wQ$. In particular, $\dim_\Mk U=1$
 for every simple $\bA$-module $U$.
  Recall first that every hereditary pure noetherian algebra of this form is isomorphic to
 a direct product of algebras of type $\wQ_n$, where $\kQ_n$ is a cycle
  \[
  \xymatrix{ 1 \ar[r] & 2 \ar[r] & \dots & n \ar@{<-}[l] \ar@/^1pc/[lll] },
 \]

\medskip\noindent
 or, equivalently, subalgebras $\bH_n$ in $\Mat(n,\bS)$, where $\bS=\Mk[[t]]$,
 consisting of all matrices $(a_{ij})$ such that $a_{ij}(0)=0$ for $i<j$.
 If $\bA$ satisfies the conditions of Definition~\ref{noddef}, then the algebra
 $\bH$  is Morita-equivalent to an algebra of this form. Let $\bJ=\rad\bA
 =\rad\bH$. Note that $\bH/\bJ\simeq\bH\*_\bA(\bA/\bJ)$ as left $\bH$-module.

 If $\bA$ is local, then $d=\dim_\Mk\bH/\bJ\le2$. If $d=1$, $\bA=\bH\simeq\bS$.
 If $d=2$, then either $\bH\simeq\bS\xx\bS$ or $\bH\simeq\bH_2$. In both cases
 $\bH/\bJ\simeq\Mk\xx\Mk$ and $\bA/\bJ\simeq\Mk$ can be embedded into $\bH/\bJ$
 only diagonally. Therefore, in the former case $\bA$ is identified with the subalgebra in $\bS\xx\bS$
 consisting of all pairs $(a,b)$ such that $a(0)=b(0)$, i.e. $\bA\simeq\Mk[[x,y]]/(xy)$ (take $(t,0)$ for
 $x$ and $(0,t)$ for $y$). In the latter case $\bA$ is identified with the subalgebra in $\bH_2$
 consisting of  matrices $(a_{ij})$ such that $a_{11}(0)=a_{22}(0)$, i.e.
 $\bA\simeq\langle x,y\rangle/(x^2,y^2)$ (take $e_{21}$ for $x$ and $te_{12}$ for $y$).

 If $\bA$ is two-point, i.e. $\bA/\bJ\simeq\Mk^2$, then $d=\dim_\Mk\bH/\bJ\le 4$. Note
 that if $d=2$, then $\bA=\bH\simeq\wQ_2$. So we can assume that
 $d=3$ or $4$. There are the following possibilities (taking into account that
 $\bA$ is connected):

 {\sc Case 1. } $\bH=\Mat(2,\bS)$. Then $\bH/\bJ\simeq\Mat(2,\Mk)$. Any
 subalgebra of $\Mat(2,\Mk)$ isomorphic to $\Mk^2$ is conjugate to the subalgebra of
 diagonal matrices. Therefore, $\bA$ is isomorphic to the subalgebra of $\Mat(2,\bS)$
 consisting of matrices $(a_{ij})$ such that $a_{12}(0)=a_{21}(0)=0$, i.e. to the algebra (9)
 from Table~\ref{table_nodal} (take $te_{11}$ for $a$, $te_{22}$ for $b$, $te_{21}$ for $c$
 and $te_{12}$ for $d$).

 {\sc Case 2. } $\bH=\bH_3$. Then $\bH/\bJ\simeq\Mk^3$ and the embedding
 $\Mk^2\to\Mk^3$ (up to a permutation of components) maps $(\al,\be)$ to $(\al,\al,\be)$.
 Therefore, $\bA$ is isomorphic to the subalgebra of $\bH_3$ consisting of  matrices
 $(a_{ij})$ such that $a_{ii}(0)=a_{jj}(0)$ for some choice of two different indices $i,j\in\{1,2,3\}$.
 One can check that all choices lead to isomorphic algebras, namely, to the algebra (4) from
 Table~\ref{table_nodal} (if $i=1,\,j=2$, take $e_{21}$ for $a$, $e_{32}$ for $b$ and $te_{13}$ for $c$).

 {\sc Case 3. } $\bH=\bS\xx\bH_2$. Again $\bH/\bJ\simeq\Mk^3$ and the embedding
 $\Mk^2\to\Mk^3$ (up to a permutation of components) maps $(\al,\be)$ to $(\al,\al,\be)$.
 Therefore, $\bA$ is isomorphic to the subalgebra of $\bH$ consisting of all  pairs $(a,(b_{ij}))$ such
 that $a(0)=b_{ii}(0)$ for some $i\in\{1,2\}$. Again both choices lead to isomorphic algebras,
 namely, to the algebra (5) from Table~\ref{table_nodal} (if $i=1$, take the pair $(t,0)$ for $a$, $(0,e_{21})$
 for $b$ and $(0,te_{12})$ for $c$).

 {\sc Case 4. } $\bH=\bH_4$. Then $\bH\simeq\Mk^4$ and the embedding
 $\Mk^2\to\Mk^4$ (up to a permutation of components) maps $(\al,\be)$ to $(\al,\al,\be,\be)$
 or to $(\al,\al,\al,\be)$. The latter case is impossible, since the length of $\bH\*_\bA U$ equals $3$,
 where $U$ is the simple $\bA$-module on which the first component of $\Mk^2$ acts nontrivially.
 Hence, to define $\bA$ up to an isomorphism we need to choose an index $k\in\{2,3,4\}$; then $\bA$
 is isomorphic to the subalgebra of $\bH$ consisting of all  matrices $(a_{ij})$ such
 that $a_{11}(0)=a_{kk}(0)$ and $a_{ii}(0)=a_{jj}(0)$, where $\{1,2,3,4\}=\{1,k,i,j\}$. One easily
 sees that the choices $k=2$ and $k=4$ lead to isomorphic algebras, and they are isomorphic to
 the algebra (6) from Table~\ref{table_nodal} (for $k=2$ take $e_{21}$ for $a$, $e_{43}$ for $b$, $e_{32}$
 for c and $te_{14}$ for $d$). The case $k=3$ gives the algebra (3) from Table~\ref{table_nodal} (take $te_{14}$
 for $a$, $e_{32}$ for $b$, $e_{43}$ for $c$ and $e_{21}$ for $d$).

 {\sc Case 5. } $\bH=\bS\xx\bH_3$. The same considerations as in Case 4 show that $\bA$ is
  isomorphic to the subalgebra of $\bH$ consisting of all  pairs $(a,(b_{ij}))$ such
 that $a(0)=b_{11}(0)$ and $a_{33}(0)=a_{22}(0)$, i.e. to the algebra (7) from Table~\ref{table_nodal}
 (take the pair $(0,e_{32})$ for $a$, $(t,0)$ for $b$, $(0,te_{13})$ for c and $(0,e_{21})$ for $d$).

 {\sc Case 6. } $\bH=\bH_2\xx\bH_2$. It follows, as above, that $\bA$ is isomorphic to the
 subalgebra of $\bH$ consisting of all pairs $((a_{ij}),(b_{ij}))$ such that $a_{ii}(0)=b_{ii}(0)$ for $i=1,2$,
 i.e. to the algebra (2) from Table~\ref{table_nodal} (take the pair $(0,te_{12})$ for $a$, $(te_{12},0)$ for $b$,
 $(e_{21},0)$ for c and $(0,e_{21})$ for $d$).

 {\sc Case 7. } $\bH=\bS\xx\bS\xx\bH_2$. Then $\bA$ is isomorphic to the subalgebra of $\bH$
 consisting of  triples $(a,b,(c_{ij}))$ such that $a(0)=c_{11}(0)$ and $b(0)=c_{22}(0)$, i.e.
 to the algebra (8) from Table~\ref{table_nodal} (take the triple $(t,0,0)$ for $a$, $(0,t,0)$ for $b$,
 $(0,0,te_{12})$ for c and $(0,0,e_{21})$ for $d$).

\end{proof}

\subsection{Gentle algebras}\label{s23}

Let $\kQ$ be a quiver and $\kI$ an admissible ideal in the path algebra  $\Mk\kQ$.
\begin{defin}
 \label{specialbi}
The pair $(\kQ,\kI)$ is said to be \emph{special biserial}  if
the following holds:
\begin{itemize}
\item[(G1)] At every vertex of $\kQ$ at most two arrows end and at most two
arrows start.
\item[(G2)] For each arrow $b$ there is at most one arrow $a$ with
$\mathfrak{t}(a)=\mathfrak{s}(b)$ and $ba\not\in \kI$ and at most one arrow $c$ with
$\mathfrak{t}(b)=\mathfrak{s}(c)$ and $cb\not\in \kI$.
\end{itemize}
\end{defin}

\begin{defin}\label{defgentle}
The pair $(\kQ,\kI)$ is said to be \emph{gentle}  if
it is special biserial, and moreover the following holds:

\begin{itemize}
\item[(G3)] $\kI$ is generated by zero relations of length 2.
\item[(G4)] For each arrow $b$ there is at most one arrow $a$ with
$\mathfrak{t}(a)=\mathfrak{s}(b)$ and $ba\in \kI$ and at most one arrow $c$ with
$\mathfrak{t}(b)=\mathfrak{s}(c)$ and $cb\in \kI$.
\end{itemize}
\end{defin}

\begin{defin}\label{defgentlealg}
A $\Mk$-algebra $\bA$ is called \emph{special biserial} (respectively, \emph{gentle}),
if it is Morita equivalent to the completion of an  algebra $\Mk\kQ/\kI$,
where the pair $(\kQ,\kI)$ is special biserial (respectively, gentle).

\end{defin}

\begin{erem}
Note that Definitions~\ref{specialbi}, \ref{defgentle}, \ref{defgentlealg} do not require the finite dimensionality of the algebra $\bA$. In the finite dimensional case special biserial algebras were defined in \cite{sw}, while gentle algebras were defined in \cite{as}. Also note that gentle algebras without completion appeared in \cite{bh} under the name \emph{locally gentle} algebras.

\end{erem}

The proof of the following statement is straightforward.

\begin{prop}\label{gentle}
(1) Let  ${\bA}$ be a complete local algebra over $\Mk$. Then $\bA $ is gentle  if
and only if $\bA$ is isomorphic to one of the following algebras:
\begin{itemize}
\item $\bL_1= \Mk.$
\item $\bL_2=\Mk[x]/(x^{2}).$
\item  $\bL_3=\Mk [[x]]$.
\item  $\bL_4=\Mk [[x,y]]/(xy)$.
\item  $\bL_5=\Mk \left<\left< x,y\right>\right> /(x^{2}, y^{2})$.
\end{itemize}

\medskip

(2) Let $\Mk\kQ/\kI$ be a two-point algebra over  $\Mk$ and $\bA$ its completion.
Then the following conditions are equivalent:

\begin{itemize}
\item $\bA$ is gentle.
\item $\bA$ is isomorphic to one of the algebras from  Table~\ref{table_gentle} below.
\end{itemize}
\end{prop}

\begin{table}[ht]
\caption{Gentle two-point algebras}\label{table_gentle}
\renewcommand\arraystretch{1.5}
\noindent\[
\begin{array}{|c|c|}
\hline
\kQ 1: \xymatrix{
1 \ar[r]^{a}&2
}&\kQ 2:\xymatrix{
1 \ar@/^1pc/[r]^{a} \ar@/_1pc/[r]_{b} &2
}\\
(1)\hspace{10pt} \kI=0&(2)\hspace{10pt}\kI=0\\
\hline
\kQ 3: \xymatrix{
1 \ar@/^1pc/[r]^{a}  &2 \ar@/^1pc/[l]_{b}
}&\kQ 4: \xymatrix{
1 \ar@/^1pc/[r]^{b}  \ar@/^2pc/[r]^{a} &2 \ar@/^1pc/[l]^{c}
}\\
(3)\hspace{10pt} \kI=0&(6)\hspace{10pt}\kI=\left<ca, bc\right>\\
(4)\hspace{10pt} \kI=\left<ba\right>&(7)\hspace{10pt}\kI=\left<ca, ac\right>\\
(5)\hspace{10pt} \kI=\left<ba, ab\right>&\\
\hline
\kQ 5: \xymatrix{
1 \ar@/^1pc/[r]^{b}  \ar@/^2pc/[r]^{a} &2 \ar@/^1pc/[l]^{c} \ar@/^2pc/[l]^{d}
}&\kQ 6: \xymatrix{
1 \ar[r]^{b} \ar@(ul,ur)[]^{a}  &2
}\\
(8)\hspace{10pt} \kI=\left<ca, db, ac, bd\right>&(10)\hspace{10pt}\kI=\left<a^{2}\right>\\
(9)\hspace{10pt} \kI=\left<ca, db, bc, ad\right>&(11)\hspace{10pt}\kI=\left<ba\right>\\
\hline
\kQ 7: \xymatrix{
1 \ar[r]^{b}  &2 \ar@(ul,ur)[]^{a}
}&\kQ 8: \xymatrix{
1 \ar@/^1pc/[r]^{b} \ar@(ul,dl)[]_{a}  &2 \ar@/^1pc/[l]^{c}
}\\
(12)\hspace{10pt} \kI=\left<a^{2}\right>&(14)\hspace{10pt}\kI=\left<a^{2}, bc\right>\\
(13)\hspace{10pt} \kI=\left<ab\right>&(15)\hspace{10pt}\kI=\left<ba, ac\right>\\
&(16)\hspace{10pt}\kI=\left<a^{2}, bc, cb\right>\\
&(17)\hspace{10pt}\kI=\left<ba, ac, cb\right>\\
\hline
\kQ 9: \xymatrix{
1 \ar[r]^{c} \ar@(ul,ur)[]^{a}  &2 \ar@(ul,ur)[]^{b}
}&\kQ 10: \xymatrix{
1 \ar@/^1pc/[r]^{c} \ar@(ul,dl)[]_{a}  &2 \ar@/^1pc/[l]^{d} \ar@(ur,dr)[]^{b}
}\\
(18)\hspace{10pt} \kI=\left<a^{2}, b^{2}\right>&(22)\hspace{10pt}\kI=\left<a^{2}, b^{2}, dc, cd\right>\\
(19)\hspace{10pt} \kI=\left<a^{2}, bc\right>&(23)\hspace{10pt}\kI=\left<a^{2}, db, bc, cd\right>\\
(20)\hspace{10pt} \kI=\left<ca, b^{2}\right>&(24)\hspace{10pt}\kI=\left<ca, db, bc, ad\right>\\
(21)\hspace{10pt} \kI=\left<ca, bc\right>&\\
\hline
\end{array}
\]
\end{table}

\begin{erem} Note that algebras $(3)$, $(8)$, $(9)$, $(14)$, $(15)$ and $(22)-(24)$ from Table~\ref{table_gentle} are nodal.
Note also that algebras $(7)$, $(11)$, $(13)$, $(17)$ and $(19)-(21)$ are infinite dimensional but not nodal.
\end{erem}

It was shown in \cite{ps} that any finite dimensional gentle algebra is derived tame. The proof of this result from \cite{ps} can not be adapted for the case of infinite dimensional gentle algebras. On the other hand, in \cite{bm} a different approach was used to obtain a classification of indecomposable objects in derived  categories over finite dimensional gentle algebras. This approach is based on the reduction of the classification problem to
 a matrix problem considered by Bondarenko in
\cite{bo}. We note that with minor modifications the same reduction works in the case of infinite dimensional gentle algebras. Hence we immediately obtain the following result.

\begin{theorem}\label{thm-gentle-tame}
Any gentle algebra is derived tame.

\end{theorem}

\subsection{Two deformations of gentle algebras}\label{s24}

Consider  the following quiver $\kQ$:

\[
\xymatrix{
1 \ar[r]^{c} \ar@(ul,ur)[]^{a}  &2 \ar@(ul,ur)[]^{b}
}
\]

Let $\bD_i=\Mk {\kQ}/\kI_i$, $i=1,2$, where
$\kI_1=\left<a^2, ca-bc\right>$ and
$\kI_2=\left<b^2, ca-bc\right>$. These two algebras are anti-isomorphic.

Consider $\bA_{\lambda}^1 =\Mk {\kQ}/\kI_1$, $\bA_{\lambda}^2 =\Mk {\kQ}/\kI_2$, where
$\kI_1=<a^2, bc-\lambda ca>$, $\kI_2=<b^2, ca-\lambda bc>$.
Note that $\bA_{\lambda}^1$ is a deformation of $(20)$, while $\bA_{\lambda}^2$ is a deformation of $(21)$ from Table~\ref{table_gentle}. Clearly, $\bA_{\lambda}^1\simeq \bD_1$ and $\bA_{\lambda}^2\simeq\bD_2$ for any $\lambda\neq 0$.

\begin{lemma}\label{def}
Algebras $\bD_1$ and $\bD_2$ are derived tame.
\end{lemma}

\begin{proof} Let $\bA=\begin{bmatrix}\bS & t\bS\\t\bS &\bS\end{bmatrix}$, where $\bS=\Mk[[t]]$.

Then we have the following short exact sequence:

\[
\xymatrix{
0 \ar[r] & P_1 \ar[r]^{t} & P_2 \ar[r] & L\ar[r] & 0
}
\]

\noindent where $P_1=\begin{bmatrix}\bS \\t\bS \end{bmatrix}$, $P_2=\begin{bmatrix} t\bS\\\bS\end{bmatrix}$ are
indecomposable projective left $\bA$-modules and $L=\begin{bmatrix}0\\\frac{\bS}{t^{2}\bS}\end{bmatrix}$.

Define a complex $T_\bp=T_0\oplus T_1$ of $\bA$-modules as follows. Let $T_0:0\to L\to 0$ (in degree $-1$) and
$T_1:0\to P_1\to 0$ (in degree $0$). It is easy to check that the complex $T_\bp$ is tilting (see \cite{rk} for definition) and the
endomorphism algebra $\End_{\kD^{b}(\md{\bA})}(T_\bp)$ is isomorphic to  $\bD_1$.
Since $\bA$ is isomorphic to the algebra (9) from the Table~\ref{table_nodal}, it is derived tame by \cite{bud}.
Therefore algebra $\bD_1$ is also derived tame. The case of the algebra $\bD_2$ is similar.
\end{proof}

\begin{erem}
 Both algebras  $(20)$ and $(21)$ are derived tame by \cite{bm}. It is known that in the finite dimensional case the tameness of  an algebra implies the tameness of its deformations \cite{d3}. But it is an open question in the infinite dimensional case.

\end{erem}

\section{Classification}\label{s3}

\subsection{Derived wildness}\label{s31}

We will need the following  hereditary algebras which are used in the
next sections.

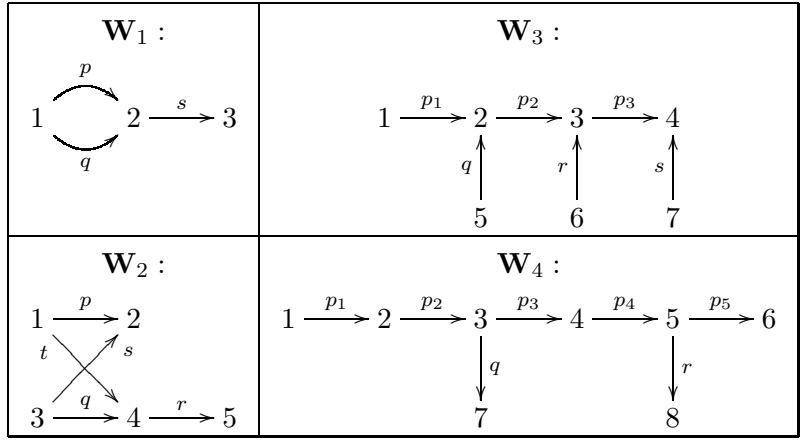
\begin{table}[ht]
\caption{Some wild hereditary algebras}\label{table_wild}
\renewcommand\arraystretch{1.5}
\noindent\[
\begin{array}{|c|c|}
\hline
\bW_1:&\bW_3:\\
\xymatrix{
1 \ar@/^1pc/[r]^{p} \ar@/_1pc/[r]_{q} &2 \ar[r]^{s} & 3
}&\xymatrix{
1 \ar[r]^{p_1} & 2 \ar[r]^{p_2} & 3 \ar[r]^{p_3} & 4   \\
& 5 \ar[u]^{q} & 6 \ar[u]^{r} & 7\ar[u]^{s}
}\\
\hline
\bW_2:&\bW_4:\\
\xymatrix{
 1 \ar[r]^{p} \ar[dr]_(.2){t} & 2 & \\
 3 \ar[r]^{q} \ar[ur]_(.8){s} & 4 \ar[r]^{r}  & 5
}&\xymatrix{
1 \ar[r]^{p_1} & 2 \ar[r]^{p_2} & 3 \ar[r]^{p_3}\ar[d]^{q} & 4 \ar[r]^{p_4}& 5  \ar[r]^{p_5}\ar[d]^{r}& 6\\
& & 7 & & 8   &
}\\
\hline
\end{array}
\]
\end{table}

It is well known that the algebras $\bW_1-\bW_4$ are wild \cite{n}.

We also need the following boxes  (see \cite{d0} for definition), which will be used in
the proof of Theorem~A and Theorem~B:

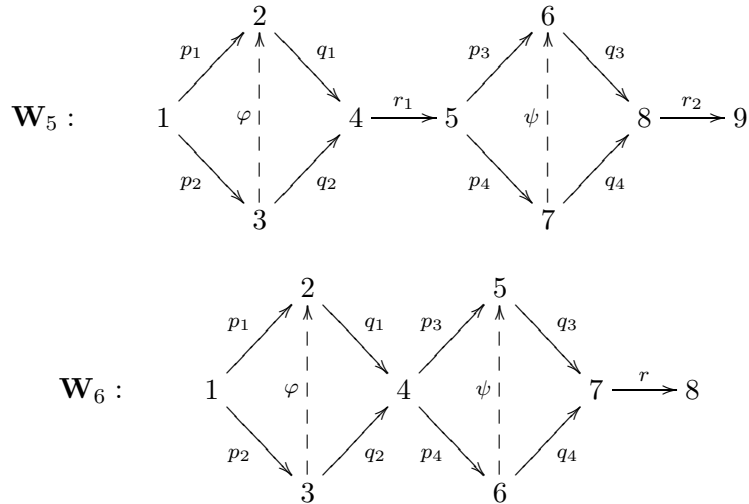
\[
 \xymatrix{
 & &  2 \ar[rd]^{q_1} & & & 6 \ar[rd]^{q_3} & &  \\
\bW_{5}:& 1 \ar[rd]_{p_2} \ar[ur]^{p_1} & & 4 \ar[r]^{r_1}& 5 \ar[rd]_{p_4} \ar[ur]^{p_3}&  & 8 \ar[r]^{r_2} & 9\\
 & &  3 \ar[ur]_{q_2}\ar@{-->}[uu]^{\varphi} & & & 7 \ar[ur]_{q_4}\ar@{-->}[uu]^{\psi}& &
 }
\]

\[
 \xymatrix{
 & &  2 \ar[rd]^{q_1} &  & 5 \ar[rd]^{q_3} & &  \\
\bW_{6}:& 1 \ar[rd]_{p_2} \ar[ur]^{p_1} & & 4  \ar[rd]_{p_4} \ar[ur]^{p_3}&  & 7 \ar[r]^{r} & 8\\
 & &  3 \ar[ur]_{q_2} \ar@{-->}[uu]^{\varphi}&  & 6 \ar[ur]_{q_4}\ar@{-->}[uu]^{\psi}& &
 }
\]

Let $f$  be the quadratic form corresponding to the box $\bW_5$ (resp., $\bW_6$)
(see \cite{d0} for definition). Consider the following dimension vector ${\bf
d}=(d_i)_{i=1}^{9}=(2,2,2,4,4,2,2,2,1)$ (resp., ${\bf
d}=(d_i)_{i=1}^{8}=(2,2,2,4,2,2,2,1)$). Since $f({\bf d})=-1$,
it follows from \cite{d0} that $\bW_5$ and $\bW_6$ are wild.

We will use the following notations. Let $\bB$ be one of the
algebras $\bW_1-\bW_4$ or one of  the boxes $\bW_5-\bW_6$. Since $\bB$ is wild, there exists
$\bB$-$\Mk\left< x,y\right>$-bimodule $M=M(\bB)$, finitely generated and free
 over $\kxy$ such that the functor $M\otimes_{\Mk\left< x,y\right>}\_$,
from the category of
finite dimensional $\Mk\left< x,y\right>$-modules to the category
of $\bB$-modules,
preserves indecomposability and
isomorphism classes. We denote by $d_i^{M}$ the rank of $M(i)$ over $\kxy$.

From now on let $\bA$ be the completion of an algebra $\Mk\kQ/\kI$ for
some finite quiver $\kQ$ and some admissible ideal $\kI$. We denote by $A_i$ the
 indecomposbale projective $\bA$-module corresponding to the vertex $i$ of $\kQ$
 and set $\tA_i=A_i\*_\Mk\kxy$.

The following technical lemmas are needed for the proof.

\begin{lemma}\label{subalg}
Let  $\bB$ be a full subalgebra of $\bA $
(i.e., a subalgebra of the form $e\bA e$ for some idempotent $e$).
If $\bB$ is derived wild then $\bA$ is derived wild.
\end{lemma}

\begin{proof}
Obvious.
\end{proof}

\begin{lemma}\label{wildpaths}
Suppose that there exist
$a,b\in \kQ_1$ and $w=\sum_i \la_iw_i\ne 0$, where $w_i$ are some paths of length $\geq 1$, such that
$\mathfrak{s}(w_i)=\mathfrak{s}(w_j)$ and $\mathfrak{t}(w_i)=\mathfrak{t}(w_j)$ for all $i,j,$ $\la_i\in
\Mk$, $\mathfrak{s}(a)=\mathfrak{s}(b)$, $\mathfrak{t}(a)=\mathfrak{t}(b)$, $\mathfrak{t}(a)=\mathfrak{s}(w)$ (resp. $\mathfrak{s}(a)=\mathfrak{t}(w)$)
and $wa, wb\in \kI$ (resp., $aw, bw\in\kI$).
Then  $\bA $ is  derived wild.
\end{lemma}

\begin{proof}
We assume that $\mathfrak{s}(a)=\mathfrak{t}(w)$ (the other case is similar).
Let $M=M(\bW_1)$. Denote by $N_\bp$ the following complex of $\bA
-\Mk\left< x,y\right>$-bimodules:

\[
\xymatrix{
d_1\tA_{\mathfrak{t}(a)} \ar@/^1pc/[r]^{aM(p)} \ar@/_1pc/[r]_{bM(q)} &d_2\tA_{\mathfrak{s}(a)} \ar[r]^{wM(s)} & d_3\tA_{\mathfrak{s}(w)}
}
\]

\noindent or, equivalently,

\[
\xymatrix{
\cdots \ar[r] & 0 \ar[r] & d_1\tA_{\mathfrak{t}(a)} \ar[rr]^{aM(p)+bM(q)} & &d_2\tA_{\mathfrak{s}(a)} \ar[r]^{wM(s)} & d_3\tA_{\mathfrak{s}(w)} \ar[r] & 0 \ar[r] & \cdots
}
\]

It is not difficult to verify that the functor
$N_\bp\otimes_{\Mk\left< x,y\right>}-$, which acts from the category of
finite dimensional $\Mk\left< x,y\right>$-modules to the category
$\kP_{\min}(\bA)$, preserves indecomposability and the isomorphism
classes. Hence, $\bA $ is derived wild.
\end{proof}

\begin{lemma}\label{wildloop}
Suppose that there exist $a,b\in \kQ_1$ such that $\mathfrak{s}(a)=\mathfrak{t}(a)=\mathfrak{t}(b)$
(resp., $\mathfrak{s}(a)=\mathfrak{t}(a)=\mathfrak{s}(b)$) and $a^{2}, ab\in \kI$ (resp., $a^{2}, ba\in \kI$).
Then $\bA$ is  derived wild.
\end{lemma}

\begin{proof}
We assume that $\mathfrak{s}(a)=\mathfrak{t}(a)=\mathfrak{s}(b)$ (the other case is similar).
Let $M=M(\bW_3)$ be as above.
Let us denote by $N_\bp$ the following complex of $\bA$-$\Mk\left< x,y\right>$-bimodules.

\[
 \xymatrix{
 d_1\tA_{\mathfrak{s}(a)} \ar[r]^{aM(p_1)} &  d_2\tA_{\mathfrak{s}(a)} \ar[r]^{aM(p_2)} & d_3\tA_{\mathfrak{s}(a)} \ar[r]^{aM(p_3)} & d_4\tA_{\mathfrak{s}(a)} \\
 d_5\tA_{\mathfrak{t}(b)} \ar[ur]^{bM(q)} &  d_6\tA_{\mathfrak{t}(b)} \ar[ur]^{bM(r)} & d_7\tA_{\mathfrak{t}(b)} \ar[ur]^{bM(s)} &
 }
\]

Here each column presents direct summands of a non-zero component $N_{n}$ (in our case $n=3,2,1,0$) and
the arrows show the non-zero components of the differential.
Again applying the functor
$N_\bp\otimes_{\Mk\left< x,y\right>}-$ we immediately obtain that $\bA $ is derived wild.
\end{proof}

\subsection{Proof of Theorem A}\label{s32}

\begin{proof}
\,\,\,$"\Rightarrow."$

 Suppose first that $\bA$ is pure noetherian. Since $\bA$ is derived
 tame, it is tame and hence nodal by \cite{d1}. Then it follows from Proposition~\ref{nodal} that
$\bA$ is isomorphic to one of the algebras $\bL_3-\bL_5$.

 Suppose now that $\bA$ has some minimal ideal $\kJ$.
 If $\kQ_1=\emptyset$ then $\bA$ is isomorphic to the algebra $\bL_1$.
Suppose that there exist $a,b\in \kQ_1$, $a\ne b$. Consider any $0\ne z\in \kJ$.
Then $\bA$ satisfies the conditions of
Lemma~\ref{wildpaths}, where $a=a, b=b$ and $w=z$, hence $\bA$ is
derived wild.

Therefore we can assume that $\kQ_1$  has only one arrow, say $a$. Then $a^{n}\in \kJ$
for some $n\in \mN, n>1$. If $n=2$ then $\bA$ is isomorphic to  $\bL_2$.
Assume that $n>2$.
Let $M=M(\bW_6)$. Denote by $N_\bp$ the following complex of $\bA$-$\Mk\left< x,y\right>$-bimodules:

\[ \hspace*{-2em}
 \xymatrix{
 &  d_2\tA \ar[rd]|{a^{n-1}M(q_1)} & &  d_5\tA \ar[rd]^{a^{n-1}M(q_3)}& & & \\
 d_1\tA \ar[rd]_{a^{n-1}M(p_2)} \ar[ur]^{a^{n}M(p_1)} & & d_4\tA \ar[rd]|{a^{n-1}M(p_4)} \ar[ur]|{a^{n}M(p_3)}&  & d_7\tA \ar[rr]^{a^{n}M(r)} && d_8\tA\\
 &  d_3\tA \ar[ur]|{a^{n}M(q_2)} & & d_6\tA \ar[ur]_{a^{n}M(q_4)}& & &
 }
\]

Again it is easy to check that the functor
$N_\bp\otimes_{\Mk\left< x,y\right>}-$  preserves indecomposability and the isomorphism
classes. We conclude that $\bA $ is derived wild.

\medskip
$"\Leftarrow."$  Since $\bL_1$ and $\bL_3$ are hereditary, it follows from \cite{h}
that $\bL_1$ is derived finite and $\bL_3$ is derived discrete but not derived finite.
Since $\bL_2$ is gentle, it follows from \cite{bm} that $\bL_2$ is
derived discrete but not derived finite.
Since $\bL_4$ and $\bL_5$ are nodal algebras, it follows from \cite{bud} that $\bL_4$ and $\bL_5$
are derived tame but not derived discrete.
\end{proof}

\subsection{Proof of Theorem B}\label{s33}

\begin{proof}
(1) $(i)\Rightarrow (iii).$

\medskip
Since $\bA$ is derived tame, then $\bA $ is tame and hence $\Mk\kQ/\rad^2(\Mk\kQ)$ is tame.
Then we conclude  that $\kQ$ is one of the quivers from Table~\ref{table_gentle}.

Let us consider all cases.

\medskip
{\sc Case 1. }  $\kQ=\kQ1$.
Then $\bA$ is isomorphic to the algebra (1) from Table~\ref{table_gentle}.

\medskip
{\sc Case 2. }  $\kQ=\kQ2$.
Then $\bA$ is isomorphic to the algebra (2) from Table~\ref{table_gentle}.

\medskip
{\sc Case 3. }  $\kQ=\kQ3$.
It follows from Theorem A and Lemma~\ref{subalg} that for  $i\in\{1,2\}$ we have $e_i\bA e_i\cong \bL_j$  for some $j\in\{1,2,3\}$.
If $e_i\bA e_i\cong \bL_3$ for some $i$, then $\bA$ is isomorphic to the algebra (3) from Table~\ref{table_gentle}.
If $e_i\bA e_i\cong \bL_1$ for  $i=1,2$, then $\bA$ is isomorphic to the algebra (5) from Table~\ref{table_gentle}.
If $e_i\bA e_i\cong \bL_1$ and $e_j\bA e_j\cong \bL_2$ for $i,j\in\{1,2\}$, then $\bA$
is isomorphic to the algebra (4) from Table~\ref{table_gentle}.
Suppose finally that  $e_i\bA e_i\cong \bL_2$ for  $i=1,2$.
Then $baba,abab\in\kI$, $ab\not\in\kI$ and $ba\not\in\kI$.
Therefore,   $aba\in \kI$ or $bab\in \kI$ or
$\kI=\left<abab,baba\right>$. Let us consider all cases.

\medskip
\noindent(a) $aba\in \kI$, $bab\not\in \kI$.

Let $M=M(\bW_5)$. Let us denote by $N_\bp$ the following complex of $\bA$-$\Mk\left< x,y\right>$-bimodules:

\[
 \xymatrix{
 &  d_2\tA_2 \ar[rd]^{aM(q_1)} & & & d_6\tA_2 \ar[rd]^{aM(q_3)}& & & \\
 d_1\tA_2 \ar[rd]_{aM(p_2)} \ar[ur]^{abM(p_1)} & & d_4\tA_1 \ar[r]^{babM(r_1)}& d_5\tA_2\ar[rd]_{aM(p_4)} \ar[ur]^{abM(p_3)}&  & d_8\tA_1 \ar[r]^{baM(r_2)} & d_9\tA_1\\
 &  d_3\tA_1 \ar[ur]_{baM(q_2)} & & & d_7\tA_1 \ar[ur]_{baM(q_4)}& & &
 }
\]

\medskip
Since the functor
$N_\bp\otimes_{\Mk\left< x,y\right>}\_$  preserves indecomposability and the isomorphism
classes, we conclude that $\bA $ is derived wild.

\medskip
\noindent(b) $bab\in \kI$, $aba\not\in \kI$. This case is similar to the case (a).

\medskip
\noindent(c) $\kI=\left<aba,bab\right>$.

Let $M=M(\bW_5)$. Let us denote by $N_\bp$ the following complex of $\bA$-$\Mk\left< x,y\right>$-bimodules:

\[
 \xymatrix{
 &  d_2\tA_1 \ar[rd]^{bM(q_1)} & & & d_6\tA_2 \ar[rd]^{aM(q_3)}& & & \\
 d_1\tA_1 \ar[rd]_{bM(p_2)} \ar[ur]^{baM(p_1)} & & d_4\tA_2 \ar[r]^{abM(r_1)}& d_5\tA_2\ar[rd]_{aM(p_4)} \ar[ur]^{abM(p_3)}&  & d_8\tA_1 \ar[r]^{baM(r_2)} & d_9\tA_1\\
 &  d_3\tA_2 \ar[ur]_{abM(q_2)} & & & d_7\tA_1 \ar[ur]_{baM(q_4)}& & &
 }
\]

\medskip
Applying the functor
$N_\bp\otimes_{\Mk\left< x,y\right>}\_$  we conclude that $\bA $ is derived wild.

\medskip
\noindent(d) $\kI=\left<abab,baba\right>$.

Let $M=M(\bW_5)$. Let us denote by $N_\bp$ the following complex of $\bA$-$\Mk\left< x,y\right>$-bimodules:

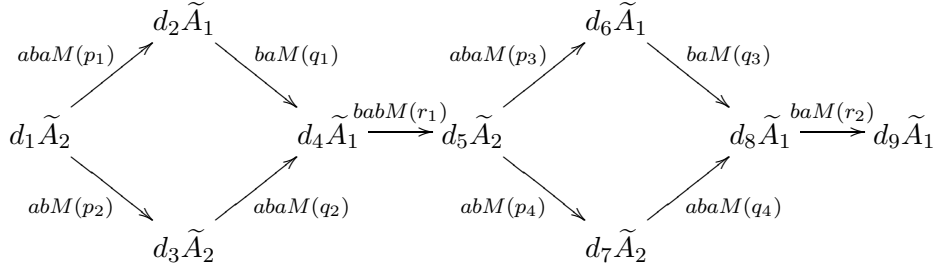
\[
 \xymatrix{
 &  d_2\tA_1 \ar[rd]^{baM(q_1)} & & & d_6\tA_1 \ar[rd]^{baM(q_3)}& & & \\
 d_1\tA_2 \ar[rd]_{abM(p_2)} \ar[ur]^{abaM(p_1)} & & d_4\tA_1 \ar[r]^{babM(r_1)}& d_5\tA_2\ar[rd]_{abM(p_4)} \ar[ur]^{abaM(p_3)}&  & d_8\tA_1 \ar[r]^{baM(r_2)} & d_9\tA_1\\
 &  d_3\tA_2 \ar[ur]_{abaM(q_2)} & & & d_7\tA_2 \ar[ur]_{abaM(q_4)}& & &
 }
\]

\medskip
Again applying the functor
$N_\bp\otimes_{\Mk\left< x,y\right>}\_$   we conclude that $\bA $ is derived wild.

\medskip
{\sc Case 4. }  $\kQ=\kQ4$.
Up to isomorphism we may assume that $ca+cf_1\in \kI$ and  $ac+f_2c\in \kI$ or $bc+f_3\in \kI$ for some
$f_i\in \rad^{2} \bA$; otherwise $\bA/{\rad^{3}\bA}$ is wild by \cite{bh} and hence $\bA$
is wild. Replacing $a+f_1$ with $a$ we can assume in both cases that $ca\in \kI$. Let us consider all cases.

\medskip
\noindent (a) $ca, ac+f_2c\in \kI$.
Then it follows from Lemma~\ref{wildpaths} that $cb\not\in\kI$.
Since $f_2=ag_1(cb)+bg_2(cb)$ for some polynomials $g_i$, we have
$ac=h(bc)$ for some polynomial $h$.
Then it follows from Lemma~\ref{subalg} that $e_1\bA e_1\cong \bL_2$, hence $cbcb\in \kI$.
If $bcbc\not\in\kI$ or $acbc\not\in\kI$, $\bA$ is derived wild by Lemma~\ref{wildpaths}.
Therefore $f_2\in\kI$ and hence $ac\in \kI$.
Hence
$\bA$ is isomorphic to the algebra (7) from Table~\ref{table_gentle}.

\medskip
\noindent (b) $ca, bc+f_3c\in \kI$.
Replacing $b+f_3$ with $b$ we can assume that $bc\in \kI$.
Then it follows from Lemma~\ref{wildpaths} that $ac\not\in\kI$, $cb\not\in\kI$ and $acb\not\in\kI$, hence
$\bA$ is isomorphic to the algebra (6) from Table~\ref{table_gentle}.

\medskip
{\sc Case 5. }  $\kQ=\kQ5$.
Suppose first that $\bA$ is pure noetherian. Since $\bA$ is derived
 tame, it is tame and hence nodal by \cite{d1}. Then it follows from Proposition \ref{nodal} that
$\bA$ is isomorphic to one of the algebras $(8)$ or $(9)$ from Table~G.

Suppose finally that $\bA$ has some minimal ideal $\kJ$.
Given $0\ne z\in \kJ$. We suppose that $s(z)=e(a)$ (the
case $s(z)=s(a)$ is similar). Then $\bA$ satisfies the conditions
of Lemma~\ref{wildpaths}, where $u=a, v=b$ and $w=z$, hence $\bA $ is
derived wild.

\medskip
{\sc Case 6. }  $\kQ=\kQ6$.
Suppose first that $e_1\bA e_1$ is finite-dimensional.
Then it follows from Theorem A
and Lemma~\ref{subalg} that $a^2\in I$. From Lemma~\ref{wildloop} we
conclude that $ba\not\in \kI$. Hence $\bA$ is isomorphic to the algebra (10) from Table~\ref{table_gentle}.

Suppose finally that $e_1\bA e_1$ is infinite-dimensional.
Then $e_1\bA e_1\cong \bL_2$.
If $ba\not\in\kI$ then $\bA$ is wild, since the finite-dimensional
algebra $\bA/\left<a^{7}, ba^{2}\right>$ is wild by \cite{hm} and hence
$\bA$ is derived wild.
Therefore $ba\in\kI$ and $\bA$ is isomorphic to the algebra (11) from Table~\ref{table_gentle}.

\medskip
{\sc Case 7. }  $\kQ=\kQ7$.
This case is dual to the previous case.
Then we obtain in this case that $\bA$ is isomorphic to one of the algebras $(12)$ or $(13)$ from Table~\ref{table_gentle}.

\medskip
{\sc Case 8. }  $\kQ=\kQ8$. Then it follows from Theorem A and Lemma~\ref{subalg}
that $e_1\bA e_1$ is isomorphic to one of the algebras $\bL_2-\bL_5$.
Then one of the following situations occur:

\medskip
\noindent(a)  $e_1\bA e_1$ is isomorphic to one of the algebras $\bL_4-\bL_5.$
Then $cb\not\in\Mk[[a]]$ (in particular, $cb\not\in\kI$).

Suppose first that $\bA$ is pure noetherian. Since $\bA$ is derived
 tame, it is tame and hence nodal by \cite{d1}. Then it follows from Proposition~\ref{nodal} that
$\bA$ is isomorphic to one of the algebras $(14)-(15)$ from Table~\ref{table_gentle}.

Suppose finally that $\bA$ has some minimal ideal $\kJ$.
Fix $0\ne z\in \kJ$.
Then one of the following situations occur:

\medskip
\noindent(aa)  $\mathfrak{t}(z)=1.$
Let $M=M(\bW_1)$ (see Section~\ref{s31}). Let us denote by $N_\bp$ the following complex of
$\bA$-$\Mk\left< x,y\right>$-bimodules.

\[
\xymatrix{
d_1\tA_1 \ar@/^1pc/[r]^{aM(p)} \ar@/_1pc/[r]_{cbM(q)} &d_2\tA_1 \ar[r]^{zM(s)} & d_3\tA_{\mathfrak{s}(z)}
}
\]

\medskip
Applying the functor
$N_\bp\otimes_{\Mk\left< x,y\right>}\_$ we conclude that $\bA $ is derived wild.

\medskip
\noindent(ab)  $\mathfrak{s}(z)=1.$ This case is dual to the case (aa).

\medskip
\noindent(ac)  $\mathfrak{s}(z)=\mathfrak{t}(z)=2.$
Since algebras $\bL_3-\bL_5$ are pure noetherian, we obtain from Theorem A and Lemma~\ref{subalg}
that algebra $e_2\bA e_2$ is isomorphic to  the algebra $\bL_2$.

Then we can assume that $z=bc$ or $bc\in \kI$ and $z=bfc$ for some $f\in\rad \bA$.

Suppose first that $z=bc$.

Let $M=M(\bW_5)$. Let us denote by $N_\bp$ the following complex of $\bA$-$\Mk\left< x,y\right>$-bimodules:

\[
 \xymatrix{
 &  d_2\tA_1 \ar[rd]^{cM(q_1)} & & & d_6\tA_2 \ar[rd]^{bM(q_3)}& & & \\
 d_1\tA_1 \ar[rd]_{cM(p_2)} \ar[ur]^{cbM(p_1)} & & d_4\tA_2 \ar[r]^{bcM(r_1)}& d_5\tA_2\ar[rd]_{bM(p_4)} \ar[ur]^{bcM(p_3)}&  & d_8\tA_1 \ar[r]^{cbM(r_2)} & d_9\tA_1\\
 &  d_3\tA_2 \ar[ur]_{bcM(q_2)} & & & d_7\tA_1 \ar[ur]_{cbM(q_4)}& & &
 }
\]

\medskip
It shows that $\bA $ is derived wild.

Suppose finally that $bc\in\kI$ and $z=bfc$ for some $f\in\rad \bA$.

Let $M=M(\bW_5)$. Let us denote by $N_\bp$ the following complex of $\bA$-$\Mk\left< x,y\right>$-bimodules:

\[
 \xymatrix{
 &  d_2\tA_1 \ar[rd]^{fcM(q_1)} & & & d_6\tA_2 \ar[rd]^{bM(q_3)}& & & \\
 d_1\tA_1 \ar[rd]_{cM(p_2)} \ar[ur]^{cbM(p_1)} & & d_4\tA_2 \ar[r]^{bfcM(r_1)}& d_5\tA_2\ar[rd]_{bfM(p_4)} \ar[ur]^{bfcM(p_3)}&  & d_8\tA_1 \ar[r]^{cM(r_2)} & d_9\tA_2\\
 &  d_3\tA_2 \ar[ur]_{bfcM(q_2)} & & & d_7\tA_1 \ar[ur]_{cbM(q_4)}& & &
 }
\]

\medskip
We immediately see that $\bA $ is derived wild.

\medskip
\noindent(b) $e_1\bA e_1$ is isomorphic to $\bL_2$.
Then $a^{2}, cb\in\kI$ and hence $ba\not\in\kI$ and  $ac\not\in\kI$ by
Lemma~\ref{wildloop}. Suppose first that $bc\in\kI$ and $bac\not\in\kI$.
Then $\bA$ is isomorphic to the algebra (16) from Table~\ref{table_gentle}.

Suppose next that $bac\in\kI$.
Let $M=M(\bW_6)$. Denote by $N_\bp$ the following complex of $\bA$-$\Mk\left< x,y\right>$-bimodules:

\[
 \xymatrix{
 &  d_2\tA_1 \ar[rd]|{cM(q_1)} &  & d_5\tA_1 \ar[rd]^{cM(q_3)}& & & \\
 d_1\tA_2 \ar[rd]_{bM(p_2)} \ar[ur]^{baM(p_1)} & & d_4\tA_2\ar[rd]|{bM(p_4)} \ar[ur]|{baM(p_3)}&  & d_7\tA_2 \ar[r]^{bM(r)} & d_8\tA_1\\
 &  d_3\tA_1 \ar[ur]|{acM(q_2)} &  & d_6\tA_1 \ar[ur]_{acM(q_4)}& & &
 }
\]

\medskip
which shows that $\bA $ is derived wild.

Suppose, finally that $bc\not\in\kI$ and $bac\not\in\kI$.
If $bc=\lambda bac$  (resp., $bac=\lambda bc$) for some $\lambda\in \Mk, \lambda\neq 0$, we can
reduce this case to previous one replacing $b(1-a)$ (resp., $b(a-1)$) with $b$.
Therefore it remains to consider the case when $bc$ and $bac$ are linearly independent. But in this case
$e_2\bA e_2\cong \bL_5/(xy-yx)$, hence $\bA$ is wild by Theorem A and Lemma~\ref{subalg}.

\medskip
\noindent(c) $e_1\bA e_1$ is isomorphic to $\bL_3$.

\medskip
\noindent(ca) Suppose first that $cb\in \kI$.
If $ba\not\in\kI$ then $\bA$ is wild, since the finite-dimensional
algebra $\bA/\left<a^{7}, ba^{2},c\right>$ is wild by \cite{hm} and hence
$\bA$ is derived wild. The case $ac\not\in\kI$ is similar. Therefore we obtain that $ba, ac\in\kI$.
If $bc\not\in\kI$, then $\bA$ is isomorphic to the algebra (17) from Table~\ref{table_gentle}.
Therefore it remains to consider the case when $bc\in\kI$.

Let $M=M(\bW_4)$. Let us denote by $N_\bp$ the following complex of $\bA$-$\Mk\left< x,y\right>$-bimodules:

\[
 \xymatrix{
 d_1\tA_1 \ar[r]^{cM(p_1)} &  d_2\tA_2 \ar[r]^{bM(p_2)} &  d_3\tA_1 \ar[r]^{cM(p_3)} \ar[dr]^{aM(q)}& d_4\tA_2 \ar[r]^{bM(p_4)} & d_5\tA_1 \ar[r]^{cM(p_5)} \ar[dr]^{aM(r)}&  d_6\tA_2\\
 && &d_7\tA_1  & & d_8\tA_1
 }
\]

\medskip
 We conclude that $\bA $ is derived wild.

\medskip
\noindent(cb)  Suppose finally that $cb\not\in\kI$.
Then $cb=f(a)$ for some polynomial $f$ such that $f(0)=0$.
Then for any $w\in \rad \bA$ there exist $u,v\in  \bA$ such that $vwu=g(a)$ for some
polynomial $g$ and hence $\bA$ is pure noetherian. Since $\bA$ is derived tame, it is tame and
hence nodal by \cite{d1}. But it follows from Proposition~\ref{nodal} that if $\bA$ is nodal algebra with
quiver $\kQ=\kQ8$ than $e_1\bA e_1$ is isomorphic to one of the algebras $\bL_4-\bL_5$, hence
this case is impossible.

\medskip
{\sc Case 9. }  $\kQ=\kQ9$.
Then one of the following situations occurs:

\medskip
\noindent(a) $e_i\bA e_i$ is finite-dimensional for $i=1,2.$
Then it follows from Theorem A and Lemma~\ref{subalg} that $a^2,b^2\in \kI$.
Then we conclude from Lemma~\ref{wildloop} that $ca\not\in \kI$ and $bc\not\in \kI$ and
hence $\bA$ is isomorphic to the algebra (18) from Table~\ref{table_gentle}.

\medskip
\noindent(b) $e_i\bA e_i$ is infinite-dimensional for $i=1,2.$
Then it follows from Theorem A and Lemma~\ref{subalg} that $e_i\bA e_i\cong \bL_2$ for $i=1,2.$
If $ca\not\in \kI$ or $bc\not\in \kI$ then it follows from \cite{hm} that $\bA/\left<a^{5},b^{5}\right>$ is wild,
therefore $\bA$ is wild and hence $\bA$ is derived wild.
Hence we obtain that $ca,bc\in \kI$ and $\bA$ is isomorphic to the algebra (21) from Table~\ref{table_gentle}.

\medskip
\noindent(c) $e_1\bA e_1$ is finite-dimensional and $e_2\bA e_2$ is infinite-dimensional.
Then it follows from Theorem A and Lemma~\ref{subalg} that $a^2\in \kI$ and $e_2\bA e_2\cong \bL_2$.
Then we conclude from Lemma~\ref{wildloop} that $ca\not\in \kI$.
If $bc\not\in \kI$ or $ca-bc\not\in \kI$ then it follows from \cite{hm} that $\bA/\left<b^{5}\right>$ is wild,
therefore $\bA$ is wild and hence $\bA$ is derived wild. Hence we obtain that either $bc\in \kI$ and $\bA$
is isomorphic to the algebra (19) from Table~\ref{table_gentle} or
$ca-bc\in \kI$ and $\bA$ is isomorphic to the algebra $\bD_1$.

\medskip
\noindent(d) $e_1\bA e_1$ is infinite-dimensional and $e_2\bA e_2$ is finite-dimensional.
This case is dual to the previous case.
Then we obtain in this case that $\bA$ is isomorphic to the algebra (20) from Table~\ref{table_gentle} or
 to the algebra $\bD_2$.

\medskip
{\sc Case 10. }  $\kQ=\kQ10$.
Suppose first that $\bA$ is pure noetherian. Since $\bA$ is derived
 tame, it is tame and hence nodal by \cite{d1}. Then it follows from Proposition~\ref{nodal} that
$\bA$ is isomorphic to one of the algebras $(22)-(24)$ from Table~\ref{table_gentle} or to the algebra $(9)$ from Table~\ref{table_nodal}.

Suppose finally that $\bA$ has some minimal ideal $\kJ$
and consider $0\ne z\in \kJ$.
Assume that $\mathfrak{t}(z)=\mathfrak{s}(b)$ (the case $\mathfrak{t}(z)=\mathfrak{s}(a)$ is similar).
Let $M=M(\bW_2)$ (see Section~\ref{s31}) and denote by $N_\bp$ the following complex of
$\bA$-$\Mk\left< x,y\right>$-bimodules.

\[
\xymatrix{
d_1\tA_1 \ar[r]^{aM(p)} \ar[dr]_(.2){dM(t)} & d_2\tA_1 & \\
d_3\tA_2 \ar[r]^{bM(q)} \ar[ur]_(.8){cM(s)} & d_4\tA_2 \ar[r]^{zM(r)}  & d_5\tA_{\mathfrak{s}(z)}
}
\]

\medskip
Again we immediately
 conclude that $\bA $ is derived wild.

\medskip
\noindent$(iii) \Rightarrow (ii).$
This follows from Lemma~\ref{nodal} and Lemma~\ref{gentle}.

\medskip
\noindent$(ii) \Rightarrow (i).$
It follows from \cite{bud} that nodal algebras are derived tame.
The derived tameness of gentle algebras follows from Theorem~\ref{thm-gentle-tame} while
the derived tameness of algebras $\bD_1$ and $\bD_2$ follows from Lemma~\ref{def}.

\medskip

Statements (2) and (3) follow from \cite{bud} and \cite{bm}.

\end{proof}


\begin{thebibliography}{99}

\bibitem{as} I.~Assem, A.~Skowro\'{n}ski.
Iterated tilted algebras of type $\widetilde{\mA}_{n}$.
\emph{Math. Z.} \textbf{195} (1987), 269--290.

\bibitem{b} H.~Bass.
Finitistic dimension and a homological generalization of semi-primary rings.
\emph{Trans. Amer. Math. Soc.} \textbf{95} (1960), 466--488.

\bibitem{bd} V.~Bekkert, Yu.~Drozd.
Tame-wild dichotomy for derived categories.
arXiv:math.RT/0310352.

\bibitem{bg} K.~Bongartz, P.~Gabriel.
Covering spaces in representation theory.
\emph{Invent. Math.} \textbf{65} (1982), 331--378.

\bibitem{bo} V.~M.~Bondarenko.
Representations of dihedral groups over a field of characteristic 2.
\emph{Mat. Sb.} \textbf{96} (1975), 63--74;
English transl.: Math. USSR Sbornik {\bf 25} (1975), 58--68.

\bibitem{bod} V.~M.~Bondarenko, Yu.~A.~Drozd.
Representation type of finite groups.
\emph{Zap. Nau$\check{c}$n. Sem. LOMI} \textbf{71} (1977), 24--41.

\bibitem{br} S.~Brenner.
Decomposition properties of some small diagrams of modules.
\emph{Sympos. Math. Ist. Naz. Alta Mat.} \textbf{13} (1974), 127--141.

\bibitem{bhan} Th.~Brustle, Y.~Han.
Two-point algebras without loops.
\emph{Comm. Algebra} \textbf{29} (2001).

\bibitem{bm} V.~Bekkert, H.~Merklen.
Indecomposables in derived categories of gentle algebras.
\emph{Algebras and Representation Theory} \textbf{6} (2003), 285--302.

\bibitem{bud} I.~Burban, Yu.~Drozd.
Derived categories of nodal algebras.
\emph{Journal of Algebra} \textbf{272} (2004), 46--94.

\bibitem{bud1} I.~Burban, Yu.~Drozd.
Coherent sheaves on rational curves with simple double points and transversal intersections.
\emph{Duke Math. Journal} \textbf{121} (2004), 189--229.

\bibitem{bh} Ch.~Bessenrodt, T.~Holm.
Weighted locally gentle quivers and Cartan matrices.
\emph{Journal of Pure and Applied Algebra} \textbf{212} (2008), 204--221.

\bibitem{d} Yu.~A.~Drozd.
Representations of commutative algebras.
\emph{Funct. Anal. Appl.} \textbf{6} (1972), 286--288.

\bibitem{d0} Yu.~Drozd.
Tame and wild matrix problems.
 \emph{Representations and quadratic forms.}
 Institute of Mathematics, Kiev, 1979, 39--74.
 (English translation:
 \emph{Amer. Math. Soc. Translations} \textbf{128} (1986) 31--55.)

\bibitem{d1} Yu.~Drozd.
Finite modules over pure Noetherian algebras.
\emph{Trudy Mat. Inst. Steklov Acad. Nauk USSR} \textbf{183} (1990), 56--68.
(English translation: \emph{Proc. Steclov Inst. of Math.} \textbf{183} (1991), 97--108.)

\bibitem{d2} Yu.~Drozd.
Derived tame and derived wild algebras.
\emph{Algebra and Discrete Math.} \textbf{3} (2004), 57--74.

\bibitem{d3} Yu.~Drozd.
Semi-continuity for derived categories.
\emph{Algebras and Representation Theory} \textbf{8} (2005), 239--248.

\bibitem{dg} P.~Dr$\ddot{a}$xler, Ch.~Geiss.
On the tameness of certain 2-point algebras.
\emph{CMS Conf. Proc.} \textbf{18} (1996), 189--199.

\bibitem{g} P.~Gabriel.
The Universal Cover of a Representation-Finite Algebra.
\emph{Lecture Notes in Mathematics} \textbf{903} (1980), 68--105.

\bibitem{ge1} Ch.~Geiss.
Tame distributive two-point algebras.
\emph{CMS Conf. Proc.} \textbf{14} (1993), 193--204.

\bibitem{ge2} Ch.~Geiss.
On degeneration of tame and wild algebras.
\emph{Arch. Math.} \textbf{64} (1995), 11--16.

\bibitem{gk} Ch.~Geiss, H.~Krause.
On the notion of derived tameness.
\emph{J. Algebra and Applications} \textbf{1} (2002), 133--157.

\bibitem{gp} I.~M.~Gelfand, V.~A.~Ponomarev.
Indecomposable representations of the Lorentz group.
\emph{Usp. Mat. Nauk} \textbf{23} (1968), 3--60.

\bibitem{h} D.~Happel.
\emph{Triangulated Categories in the Representation Theory of Finite Dimensional Algebras}.
London Mathematical Society Lecture Notes Series \textbf{119}, Cambridge University Press, Cambridge, 1988.

\bibitem{han} Y.~Han.
Wild two-point algebras.
\emph{Journal of Algebra} \textbf{247} (2002), 57--77.

\bibitem{hapr} D.~Happel, C.~M.~Ringel.
The derived category of a tubular algebra.
\emph{Lecture Notes in Mathematics} \textbf{1177}
(1984), 156--180.

\bibitem{hm} M.~Hoshino, J.~Miyachi.
Tame two-point algebras.
\emph{Tsukuba J. Math.} \textbf{12} (1988), 65--93.

\bibitem{hr} A.~Heller, I.~Reiner.
Indecomposable representations.
\emph{Ill. J. Math.} \textbf{5} (1961), 314--323.

\bibitem{n} L.~A.~Nazarova.
Representations of quivers of infinite type.
\emph{Izv. Akad. Nauk
SSSR, Ser. Mat.} \textbf{37} (1973), 752--791.
(English transl.:
\emph{Math. USSR. Izv.} \textbf{7} (1973), 749--792.)

\bibitem{ps} Z.~Pogorzaly, A.~Skowro\'{n}ski.
Self-injective biserial standard algebras.
\emph{Journal of Algebra} \textbf{138} (1991), 491--504.

\bibitem{r0} C.~M.~Ringel.
The indecomposable representations of the dihedral 2-groups.
\emph{Math. Ann.} \textbf{214} (1975), 19--34.

\bibitem{r} C.~M.~Ringel.
The representation type of local algebras.
\emph{Lecture Notes in Mathematics} \textbf{488} (1975), 282--305.

\bibitem{rk} J.~Rickard.
Morita theory for derived categories.
\emph{J. London Math. Soc.} \textbf{39} (1989), 436--456.

\bibitem{sw} A.~Skowro\'{n}ski, J.~Waschb\"{u}sch.
Representation-finite biserial algebras.
\emph{J.~Reine Angew.~Math.} \textbf{345} (1983), 172--181.

\bibitem{v} D.~Vossieck.
The algebras with discrete derived category.
\emph{Journal of Algebra} \textbf{243} (2001), 168--176.

\end{thebibliography}
\end{document}